\documentclass{article}
\usepackage[utf8]{inputenc}
\usepackage{amsmath,amsthm,dsfont,pifont,amsfonts,MnSymbol}
\usepackage{faktor}
\usepackage{xfrac}  

\title{Closed and open-closed images of submetrizable spaces}
\author{Vlad Smolin\footnote{Krasovskii Institute of Mathematics and Mechanics of UB RAS, 620108, 16 Sofia Kovalevskaya street, Yekaterinburg, Russia and Ural Federal University, Mathematical Analysis Department, 620002 Ekaterinburg, Russia; \textit{e-mail address}\textup{:} SVRusl@yandex.ru}}

\theoremstyle{plain}
\newtheorem{teo}{Theorem}
\newtheorem{lemm}[teo]{Lemma}
\newtheorem{corr}[teo]{Corollary}
\newtheorem{prop}[teo]{Proposition}
\newtheorem{ques}[teo]{Question}

\theoremstyle{definition}
\newtheorem{deff}[teo]{Definition}

\newtheorem{nota}[teo]{Notation}

\begin{document}

\maketitle
\begin{abstract}
We prove that:
\begin{itemize}
    \item[1.] If a Hausdorff M-space is a continuous closed image of a submetrizable space, then it is metrizable. 
    \item[2.] A dense-in-itself open-closed image of a submetrizable space is submetrizable if and only if it is functionally Hausdorff and has a countable pseudocharacter. 
    \item[3.] Let $Y$ be a dense-in-itself space with the following property: $\forall y\in Y\ \exists Q(y) \subseteq Y\ [y \text{ is a non-isolated q-point in } Q(y)]$. If $Y$ is an open-closed image of a submetrizable space, then $Y$ is submetrizable.
    \item[4.] There exist a submetrizable space $X$, a regular hereditarily paracompact non submetrizable first-countable space $Y$, and an open-closed map $f\colon X \to Y$.
\end{itemize}
\end{abstract}

\section{Introduction}
A space X is submetrizable if X has a weaker metrizable topology, or, equivalently, if X can be mapped onto a metrizable topological space by a continuous one-to-one map.

In \cite{miz} Takemi Mizokami proved that an open perfect image of a submetrizable space is submetrizable. Mizokami also stated that an open compact image of a submetrizable space may not be submetrizable, for the proof he referred to \cite{burke}. Earlier Popov \cite{popov} constructed a perfect map $f \colon X \to Y$, where $X$ is hereditarily paracompact and submetrizable whereas $Y$ is not submetrizable. These results motivate the following question:
\begin{ques}
    Is an open-closed image of a submetrizable space also submetrizable, and if not, when is it submetrizable?
\end{ques}

In this paper we obtain the following results concerning this question:
\begin{itemize}
    \item[\ding{95}] A dense-in-itself open-closed image of a submetrizable space is submetrizable if and only if it is functionally Hausdorff and has a countable pseudocharacter. 
    \item[\ding{95}] Let $Y$ be a dense-in-itself space with the following property: $\forall y\in Y\ \exists Q(y) \subseteq Y\ [y \text{ is a non-isolated q-point in } Q(y)]$. If $Y$ is an open-closed image of a submetrizable space, then $Y$ is submetrizable.
    \item[\ding{95}] There exist a submetrizable space $X$, a regular hereditarily paracompact non submetrizable first-countable space $Y$, and an open-closed map $f\colon X \to Y$. 
\end{itemize}

Submetrizability often appears as a sufficient condition for metrizability of generalized metric spaces or generalizations of compact spaces. In general, such theorems have the following form: Suppose that a space $X$ has the property A. If $X$ is submetrizable, then $X$ is metrizable. It seems interesting to find theorems of the following form: Suppose that a space $X$ has the property A. If $X$ is an image of a submetrizable space under a map with the property B, then $X$ is metrizable.
%Additionally he showed that if $f$ is an open compact map from a submetrizable space $X$ onto a space $Y$ such that the set $\{f^{-1}(y)\ \colon\ y \in Y \text{ and } f^{-1}(y) \text{ is nontrivial}\}$ is discrete in $X$, then $Y$ is submetrizable.

In \cite{svet}  S. A. Svetlichnyi proved the following:
\begin{teo}
    Let $f \colon X \to Y$ be a map, $X$ submetrizable, and $Y$ M-space. If one of the following conditions hold, then $Y$ is metrizable:
    \begin{itemize}
        \item $f$ is almost-open and compact;
        \item $f$ is open and  boundary-compact, and the set of non-isolated points of $Y$ is at most countable;
        \item $f$ is pseudo-open and compact, and $X$ has a countable i-weight.
    \end{itemize}
\end{teo}
%Also he showed that if a paracompact space $Y$ is an open compact image of a submetrizable space, then $Y$ is submetrizable.
In addition Svetlichnyi constructed an open s-map from a normal space with a countable i-weight onto a non-metrizable compact space. We prove the following:
\begin{itemize}
    \item[\ding{95}] If a Hausdorff M-space is a continuous closed image of a submetrizable space, then it is metrizable.
\end{itemize}

In \cite{chaber} Chaber found characterizations for open compact and perfect images of submetrizable spaces:
\begin{teo}
    A space $Y$ is a perfect image of a space with a weaker metric (separable metric) topology iff $Y$ has a sequence $\{F_n\}_{n \geq 1}$ of (countable) locally finite closed covers such that $\bigcap_{n \geq 1} \mathsf{St}(y, F_n) = \{y\}$ for $y \in Y$.
\end{teo}
\begin{teo}
    A space $Y$ is an open and compact image of a space with a weaker metric (separable metric) topology iff $Y$ has a sequence of (countable) point-finite open covers $\{V_n\}_{n \geq 1}$ such that $\bigcap_{n \geq 1} \mathsf{St}(y, V_n) = \{y\}$ for $y \in Y$.
\end{teo}
From this theorems Chaber deduced corollaries about submetrizability of perfect and open compact images of submetrizable spaces:
\begin{teo}
    If a collectionwise normal (normal) and perfectly normal space $Y$ is a perfect image of a space $X$ with a weaker (separable) metric topology, then $Y$ has a weaker (separable) metric topology. 
\end{teo}
\begin{teo}
    If a collectionwise normal (normal) space $Y$ is an open and compact image of a space $X$ with a weaker (separable) metric topology, then $Y$ has a weaker (separable) metric topology. 
\end{teo}

Also many interesting results about compact images of submetrizable spaces were obtained in \cite{Yan} and \cite{Lin}. For open complete and open uniformly complete images of submetrizable spaces see \cite{Chob}.

\section{Notation and terminology}
We use terminology from \cite{enc}. Throughout this paper all topological spaces are assumed to be $T_1$, all maps are continuous and onto. 

\begin{nota} ...
    \begin{itemize}
        \item[\ding{46}\,] $\mathbb{P} \coloneq $ the set of irrational numbers;
        \item[\ding{46}\,] $\mathbb{I} \coloneq $ the interval $[0,1]$;
        \item[\ding{46}\,] $\mathbb{Q} \coloneq $ the set of rational numbers;
        \item[\ding{46}\,] $\mathbb{M} \coloneq $ the Michael line, i.e. the set $\mathbb{R}$ with the topology generated by the base $\{U\cup K\ \colon\ U \in \tau_{\mathbb{R}} \text{ and } K \subseteq \mathbb{P}\}$ (for details see page 77 in \cite{enc}).
    \end{itemize}
\end{nota}

\begin{nota} Let $X$ be a topological space. Then
    \begin{itemize}
        \item[\ding{46}\,] if $\langle x_n \rangle_{n \in \omega}$ is a sequence in $X$, then $\langle x_n \rangle_{n \in \omega} \xrightarrow{X} x$ means that $\langle x_n \rangle_{n \in \omega}$ converges to $x$ in $X$;
        \item[\ding{46}\,] $X$ is called functionally Hausdorff if for any distinct points $x, y \in X$ there exists a continuous function $f\colon X \to \mathbb{I}$ such that $f(x) = 0$ and $f(y) = 1$;
        \item [\ding{46}\,] ${g}\circ {f}$ is the composition of functions ${g}$ and ${f}$ (that is, ${g}$ after ${f}$);
        \item [\ding{46}\,] ${f} \upharpoonright A$ $\coloneq$ the restriction of the function ${f}$ to the set ${A}$;
        \item[\ding{46}\,] $\mathsf{C}(X) \coloneq$ the family of all continuous real valued functions on $X$;
        \item[\ding{46}\,] a subset of $X$ is called {\it cozero} if it is equal to $f^{-1}\big [(0,1]\big ]$ for some continuous function $f\colon X \to \mathbb{I}$; 
        \item[\ding{46}\,] a subset $A$ of $X$ is called {\it relatively pseudocompact} if $f \upharpoonright A$ is bounded for every $f \in \mathsf{C}(X)$;
        \item[\ding{46}\,] $\mathsf{coz}(X) \coloneq$ the set of cozero sets in X;
        \item[\ding{46}\,] $\tau(X)\downarrow\ \coloneq \{\sigma \subseteq \tau(X)\ \colon\ \sigma \text{ is a topology}\}$
        \item[\ding{46}\,] $\tau(X)\downarrow_{metr}\ \coloneq \{\sigma \subseteq \tau(X)\ \colon\ \sigma \text{ is a topology and } \langle X, \sigma\rangle \text{ is metrizable}\}$;
        \item[\ding{46}\,] Let $\mathfrak{F}$ be a collection of subsets of $X$. Then $\mathfrak{F}$ is called hereditarily closure-preserving (HCP) if for every subcollection $\{F_\alpha\ \colon\ \alpha \in \Gamma\} \subseteq \mathfrak{F}$ and any $B_\alpha \subseteq F_\alpha$, $\bigcup_{\alpha \in \Gamma}\mathsf{Cl}(B_\alpha) = \mathsf{Cl}(\bigcup_{\alpha \in \Gamma}B_\alpha)$;
        \item[\ding{46}\,] let $A$ be a subset of $X$. Then $\partial_{\tau(X)}A \coloneq \mathsf{Cl}(A) \cap \mathsf{Cl}(X \setminus A) = $ the {\it boundary} of $A$;
        \item[\ding{46}\,] let $A$ be a subset of $X$. Then $\tau(X) \upharpoonright A \coloneq \{U \cap A\ \colon\ U \in \tau(X) \} =$ the subspace topology on $A$. Sometimes we will write $\langle A, \tau(X) \rangle$ instead of $\langle A, \tau(X) \upharpoonright A \rangle$;
        \item[\ding{46}\,] let $x \in X$. Then $x$ is called a {\it q-point} \cite{Mich} if there exists a sequence $\langle U_i \rangle_{i \in \omega}$ of open neigbourhoods of $x$ such that, if $x_i \in U_i$ and the $x_i$ are all distinct, then $\{x_i\ \colon\ i \in \omega\}$ has an accumulation point.
    \end{itemize}
\end{nota}

\begin{nota} Let $X$ and $Y$ be a topological spaces and let $f \colon X \to Y$ be a map. Then 
    \begin{itemize}
        \item[\ding{46}\,] $f$ is called {\it open} ({\it closed}) if for every open (closed) subset $A$ of $X$ the set $f[A]$ is open (closed) in $Y$;
        \item[\ding{46}\,] $f$ is called {\it open-closed} if $f$ is open and closed;
        \item[\ding{46}\,] $f$ is called {\it perfect} if $f$ is closed and for every $y \in Y$, $f^{-1}(y)$ is compact;
        \item[\ding{46}\,] $f$ is called {\it quasi-perfect} if $f$ is closed and for every $y \in Y$, $f^{-1}(y)$ is countably compact.
    \end{itemize}
\end{nota}

The following is our main technical tool in this paper:
\begin{deff}
    Let $X$ be a topological space, $\gamma \in \tau(X)\downarrow$, and let $f$ be a continuous function from $X$ to $Y$. Then $\sigma(\gamma, f) \coloneq$ the topology on $X$ generated by the subbase $\gamma \cup \{f^{-1}[U]\ \colon\ U \in \tau(Y)\}$.
\end{deff}

\section{Known results that we need}

In this section we list the known results that we use in our proofs.

\begin{prop}[Lemma 3.1 in \cite{oka}] \label{Oka_lemma}
    Let $X$ be a submetrizable space, and let $\mathfrak{U}$ be a $\sigma$-discrete collection of cozero sets of $X$. Then there exist a metric space $M$ and a one-to-one map $f\colon X \to M$ such that $f[U]$ is an open set of $M$ for every $U \in \mathfrak{U}$.  
\end{prop}

\begin{prop} [Theorem 2.1 in \cite{Mich}] \label{Michael_theorem}
    Let $f\colon X \to Y$ be a closed surjection. If $y\in Y$ is a q-point, then $\partial_{\tau(X)}f^{-1}(y)$ is relatively pseudocompact.
\end{prop}

\begin{prop} [Lemma 4 in \cite{burke_eng}] \label{Eng_lemma}
    Suppose $p$ is a limit point of a set $A$ in a space $X$ and that there is a $G_\delta$-subset $G$ of X which contains $p$ and has $G \cap (A \setminus \{p\}) = \varnothing$. Then any HCP collection of neighbourhoods of $p$ must be finite. 
\end{prop}

\begin{prop} [Remark after Theorem 1 in \cite{miz}]
\label{miz_theorem}
    Open perfect maps preserve submetrizability. 
\end{prop}

\section{Technical lemmas}

\begin{lemm} \label{lemma_1}
    Let $X$ be submetrizable and let $F$ be a relatively pseudocompact subspace of $X$. Then $\forall \sigma_1, \sigma_2 \in \tau(X)\downarrow_{metr}[\mathsf{Cl}_{\sigma_1}(F) = \mathsf{Cl}_{\sigma_2}(F)]$.
\end{lemm}

\begin{proof}
    Assume the converse, suppose that there exist $\sigma_1, \sigma_2 \in\tau(X)\downarrow_{metr}$ such that $\mathsf{Cl}_{\sigma_1}(F) \neq \mathsf{Cl}_{\sigma_2}(F)$. We can assume without loss of generality that $\mathsf{Cl}_{\sigma_1}(F) \setminus \mathsf{Cl}_{\sigma_2}(F) \neq \varnothing$. Fix $x \in \mathsf{Cl}_{\sigma_1}(F) \setminus \mathsf{Cl}_{\sigma_2}(F)$. Since $x \nin F$, we see that there exists a nontrivial $\langle x_n \rangle_{n \in \omega} \in {}^\omega F$ such that 
    \begin{equation} \label{lemma_1_1}
        \langle x_n \rangle_{n \in \omega} \xrightarrow{\langle X, \sigma_1 \rangle} x.
    \end{equation} 
    Also take $U \in \sigma_2(x)$ such that 
    \begin{equation} \label{lemma_1_2}
        U \cap F = \varnothing.
    \end{equation}
    From Proposition \ref{Oka_lemma} it follows that there exists a $\sigma \in \tau(X)\downarrow_{metr}$ such that $\{U\}\cup\sigma_1 \subseteq \sigma$. Then from (\ref{lemma_1_1}) and (\ref{lemma_1_2}) we conclude that $\{x_n \colon n \in \omega\}$ is a closed discrete subset of $\langle X,\sigma \rangle$, and so there exists $f \in \mathsf{C}(\langle X, \sigma\rangle) \subseteq \mathsf{C}(\langle X, \tau(X)\rangle)$ such that $f \upharpoonright F$ is unbounded. This contradiction proves the lemma.
\end{proof}

\begin{lemm} \label{criter_for_rel_pseud}
    Let $X$ be submetrizable and let $F$ be a subset of $X$. Then $F$ is relatively pseudocompact if and only if $\forall \sigma \in \tau(X)\downarrow_{metr}[\mathsf{Cl}_{\sigma}(F) \text{ is a compact subset of }\\ \langle X, \sigma\rangle]$.
\end{lemm}

\begin{proof}
    Let $F$ be a relatively pseudocompact subset of $X$. Assume the converse, suppose that there exists $\sigma \in \tau(X)\downarrow_{metr}$ such that $\mathsf{Cl}_{\sigma}(F)$ is not compact in $\langle X, \sigma\rangle$. Then there exists $f \in \mathsf{C}(\langle X, \sigma\rangle)$ such that $f \upharpoonright \mathsf{Cl}_{\sigma}(F)$ is unbounded. Since $F$ is dense in $\langle \mathsf{Cl}_{\sigma}(F), \sigma \upharpoonright \mathsf{Cl}_{\sigma}(F)\rangle$, we see that $f \upharpoonright F$ is also unbounded, but $f \in \mathsf{C}(\langle X, \tau(X)\rangle)$, a contradiction. 

    Now suppose that $\forall \sigma \in \tau(X)\downarrow_{metr}[\mathsf{Cl}_{\sigma}(F) \text{ is a compact subset of } \langle X, \sigma\rangle]$. Assume the converse, suppose that there exists $f \in \mathsf{C}(\langle X, \tau(X)\rangle)$ such that $f \upharpoonright F$ is unbounded. Let $\mathfrak{U} \coloneq \{f^{-1}[(a,b)] \colon a < b \in \mathbb{Q}\}$. From Proposition \ref{Oka_lemma} it follows that there exists $\sigma \in \tau(X)\downarrow_{metr}$ such that $\mathfrak{U} \subseteq \sigma$, and so $f \in \mathsf{C}(\langle X, \sigma\rangle)$, but it contradicts the fact that $\mathsf{Cl}_{\sigma}(F) \text{ is a compact subset of } \langle X, \sigma\rangle$.
\end{proof}

\begin{lemm} \label{Utv_4}
    Let $f\colon X \to Y$ be a map, $X$ submetrizable, $Y$ functionally Hausdorff, and $y \in Y$ such that $\forall \sigma_1, \sigma_2 \in \tau(X)\downarrow_{metr}[\mathsf{Cl}_{\sigma_1}(\partial_{\tau(X)} f^{-1}(y)) = \mathsf{Cl}_{\sigma_2}(\partial_{\tau(X)} f^{-1}(y))]$. Then $\forall \sigma \in \tau(X)\downarrow_{metr}[\mathsf{Cl}_{\sigma}(\partial_{\tau(X)} f^{-1}(y)) \subseteq f^{-1}(y)]$.
\end{lemm}

\begin{proof}
    We prove that 
    \begin{equation} \label{approx_by_metr_subtop}
        \forall x \in X \setminus f^{-1}(y)\ \exists \gamma \in \tau(X)\downarrow_{metr} [x \nin \mathsf{Cl}_{\gamma}(\partial_{\tau(X)} f^{-1}(y))].
    \end{equation}

    Let $x \in X \setminus f^{-1}(y)$, then $f(x) \neq y$. Since $Y$ is functionally Hausdorff, then there exists $U \in \mathsf{coz}(Y)$ such that $f(x) \in U$ and $y \nin U$. Now we see that
    $$
        x \in f^{-1}[U];
    $$
    $$
        f^{-1}(y) \cap f^{-1}[U] = \varnothing;
    $$
    $$
        f^{-1}[U] \in \mathsf{coz}(X).
    $$
    Then from Proposition \ref{Oka_lemma} it follows that there exists $\gamma \in \tau(X)\downarrow_{metr}$ such that $x \nin \mathsf{Cl}_{\gamma}(\partial_{\tau(X)} f^{-1}(y))$. Now the statement of the Lemma follows from (\ref{approx_by_metr_subtop}).
\end{proof}

\begin{lemm} \label{Utv_2}
    Let $f\colon X \to Y$ be a map, $y \in Y$, and $\tau \in \tau(X)\downarrow$. Then $\mathsf{Cl}_{\tau}(\partial_{\sigma(\tau, f)}f^{-1}(y)) \cap f^{-1}(y) = \partial_{\sigma(\tau, f)}f^{-1}(y)$.
\end{lemm}

\begin{proof}
    Let $x \in \mathsf{Cl}_{\tau}(\partial_{\sigma(\tau, f)}f^{-1}(y)) \cap f^{-1}(y)$, and let $W$ be an open neighbourhood of $x$ in $\langle X, \sigma(\tau, f) \rangle$. Without loss of generality, we can assume that $W = U \cap f^{-1}[V]$, where $U \in \tau(x)$ and V is an open neighbourhood of $y$. Then 
    $$
        W \cap \partial_{\sigma(\tau, f)}f^{-1}(y) = U \cap \partial_{\sigma(\tau, f)}f^{-1}(y) \neq \varnothing.
    $$
    And so $x \in \partial_{\sigma(\tau, f)}f^{-1}(y)$.
\end{proof}

\begin{lemm} \label{bound_is_comp}
    Let $f\colon X \to Y$ be a closed map, $X$ submetrizable, $y$ a q-point in $Y$, and $Y$ functionally Hausdorff. Then $\forall \tau \in \tau({X}) \downarrow_{metr} [\partial_{\sigma(\tau, f)}f^{-1}(y) \text{ is a compact subset of }\\ \langle X, \sigma(\tau, f) \rangle]$. 
\end{lemm}

\begin{proof}
    Take an arbitrary $\tau \in \tau({X}) \downarrow_{metr}$. Notice that
    $$
        f \colon \langle X, \sigma(\tau, f) \rangle \to Y \text{ is continuous and closed.}
    $$
    Then from Proposition \ref{Michael_theorem} it follows that
    \begin{equation} \label{rel_psed}
        \partial_{\sigma(\tau, f)}f^{-1}(y) \text{ is relatively pseudocompact}. 
    \end{equation}
    Now from Lemmas \ref{lemma_1} and \ref{Utv_4} we can conclude that
    $$
        \mathsf{Cl}_{\tau}(\partial_{\sigma(\tau, f)}f^{-1}(y)) \subseteq f^{-1}(y).
    $$
    So from Lemma \ref{Utv_2} it follows that
    \begin{equation} \label{cl_eq}
        \mathsf{Cl}_{\tau}(\partial_{\sigma(\tau, f)}f^{-1}(y)) = \partial_{\sigma(\tau, f)}f^{-1}(y).
    \end{equation}
    From (\ref{rel_psed}), (\ref{cl_eq}) and Lemma \ref{criter_for_rel_pseud} it follows that
    $$
        \partial_{\sigma(\tau, f)}f^{-1}(y) \text{ is a compact subspace of } \langle X, \tau \rangle.
    $$
    Now, the conclusion of the lemma follows from the equation
    $$
        \sigma(\tau, f) \upharpoonright \partial_{\sigma(\tau, f)}f^{-1}(y) = \tau \upharpoonright \partial_{\sigma(\tau, f)}f^{-1}(y).
    $$
\end{proof}

\begin{lemm} \label{heart}
    Let $f\colon X \to Y$ be an open-closed map, $X$ submetrizable, $y$ a non-isolated point of $Y$ such that any HCP collection of neighbourhoods of $y$ is finite. Then $\forall \sigma \in \tau(X)\downarrow_{metr}[\mathsf{Cl}_{\sigma}(f^{-1}(y)) \text{ is a compact subset of } \langle X, \sigma\rangle]$.
\end{lemm}

\begin{proof}
    Assume the converse. Suppose that there exists $\sigma \in \tau(X)\downarrow_{metr}$ such that $\mathsf{Cl}_{\sigma}(f^{-1}(y)) \text{ is not a compact subset of } \langle X, \sigma\rangle$. Let $\{d_n\ \colon\ n \in \omega\}$ be a closed discrete subset of $\langle X, \sigma\rangle$ that is contained in $\mathsf{Cl}_{\sigma}(f^{-1}(y))$. Now let $\{U_n\ \colon\ n \in \omega\}$ be a discrete collection of open sets of $\langle X, \sigma\rangle$ such that $d_n \in U_n$ and $d_k \nin U_n$ whenever $k \neq n$. Since $f^{-1}(y)$ is dense in $\mathsf{Cl}_{\sigma}(f^{-1}(y))$, we see that $\forall n \in \omega [U_n \cap f^{-1}(y) \neq \varnothing]$. Since $Y$ is a $T_1$ space, $y$ is non-isolated, and $f[U_n]$ is an open neighbourhood of $y$ for any $n\in\omega$, we see that $f[U_n]$ is infinite for any $n \in \omega$. Now we consider two cases.
    \begin{itemize}
        \item[1.] There exists $A \in [\omega]^{\omega}$ such that $f[U_k] \neq f[U_n]$ whenever $n \neq k \in A$. Then, since the map $f$ is open-closed, $\{f[U_n]\ \colon\ n \in A\}$ is an infinite HCP collection of neighbourhoods of $y$.
        \item[2.] There exists $k \in \omega$ such that $f[U_{n_1}] = f[U_{n_2}]$ whenever $n_1, n_2 \geq k$. Consider an arbitrary infinite set $\{y_n\ \colon\ n\in\omega\} \subseteq f[U_k] \setminus \{y\}$. Then it is easy to see that $\{f[U_n]\setminus \bigcup_{i \leq n} \{y_i\}\ \colon\ n \in \omega\}$ is an infinite HCP collection of neighbourhoods of $y$.
    \end{itemize}
    This contradiction proves the Lemma.
\end{proof}

\section{Almost networks}

In this section we introduce a notion of an almost network, such structures appear naturally when we consider closed maps from submetrizable spaces onto countably compact spaces.

\begin{deff}
    Let $X$ be a topological space, and let $N$ be a family of subsets of $X$. Then $N$ is called {\it almost network} if for every point $x \in X$ and every open neighbourhood $U$ of $x$ there exists a set $F \in N$ such that $x \in F$ and $F \setminus U$ is finite.  
\end{deff}

\begin{lemm}
    If a topological space has a countable almost network, then it is Lindelof. 
\end{lemm}

\begin{lemm} \label{count_comp_is_comp}
    If a countably compact space is a closed image of a submetrizable space, then it is compact. 
\end{lemm}

\begin{proof}
    Let $X$ be a submetrizable space, $Y$ a countably compact space, and $f\colon X \to Y$ a closed map. Take $\tau \in \tau(X) \downarrow_{metr}$ and let $\mathfrak{U} = \bigcup_{n\in \omega} \mathfrak{U}_n$ be a $\sigma$-discrete base for $\tau$. Since $f$ is closed, then $\{f[U]\ \colon\ u \in \mathfrak{U}_n\}$ is an HCP collection for any $n \in \omega$. And so, since $Y$ is countably compact, $\{f[U]\ \colon\ u \in \mathfrak{U}_n\}$ is finite for any $n \in \omega$. Now we prove that $\{f[U]\ \colon\ U \in \mathfrak{U}\}$ is an almost network for $Y$.
    
    Take $y \in Y$ and $V \in \tau(Y)$ a neigbourhood of $y$. Consider $x \in X$ such that $f(x) = y$. Let $\{U_i\ \colon\ i \in \omega\}$ be a local base at $x$ in the space $\langle X, \tau \rangle$ such that $\{U_i\ \colon\ i \in \omega\} \subseteq \mathfrak{U}$. Let us prove that 
    $$
        \exists i \in \omega [f[U_i] \setminus V \text{ is finite.}]
    $$
    Assume the converse, suppose that $\forall i \in \omega [f[U_i] \setminus V \text{ is infinite}]$. Then there exists a set $\{y_i\ \colon\ i \in \omega\}$ such that 
    \begin{equation} \label{z}
        y_i \neq y_j \text{ whenever } i \neq j.
    \end{equation}
    and 
    \begin{equation} \label{zzz}
        y_i \in f[U_i] \setminus V \text{ for any } i \in \omega. 
    \end{equation}
    Let $\{x_i\ \colon\ i \in \omega\}$ be a subset of $X$ such that $x_i \in U_i$ for any $i \in \omega$ and $f(x_i) = y_i$ for any $i \in \omega$. Take a nontrivial sequence $\langle z_n \rangle_{n \in \omega}$ in $\{x_i\ \colon\ i \in \omega\}$ such that $\langle z_n \rangle_{n \in \omega} \xrightarrow{\langle X, \tau \rangle} x$. Since $f$ is continuous, from (\ref{zzz}) it follows that $\langle z_n \rangle_{n \in \omega}$ does not converge to $x$ in $\langle X, \tau(X) \rangle$. So, $\{z_n\ \colon\ n \in \omega\}$ is a closed discrete subset of $\langle X, \tau(X) \rangle$, and since $f$ is closed, from (\ref{z}) it follows that $\{f(z_n)\ \colon\ n \in \omega\}$ is an infinite closed discrete subset of $Y$, a contradiction.

    Now we see that $\{f[U]\ \colon\ U \in \mathfrak{U}\}$ is a countable almost network for $Y$, so, from previous Lemma it follows that $Y$ is Lindelof, and since it is countably compact, we see that $Y$ is compact.
\end{proof}

\section{Main results}

\begin{teo}
    Let $X$ be submetrizabe, $Y$ a Hausdorff M-space, and $f\colon X \to Y$ closed. Then $Y$ is metrizable. 
\end{teo}

\begin{proof}
    Fix a topology $\tau \in \tau(X)\downarrow_{metr}$ and notice that 
    $$
        f\colon \langle X, \sigma(\tau, f) \rangle \to Y \text{ is continuous and closed.}
    $$

    Since $Y$ is an M-space, from \cite[Theorem 3.6]{gms} it follows that there exist a metrizable space $M(Y)$ and a quasi-perfect map $g \colon Y \to M(Y)$. Now, from Lemma \ref{count_comp_is_comp} it follows that $\forall m \in M(Y)[g^{-1}(m) \text{ is compact}]$, and so $g\colon Y \to M(Y)$ is a perfect map. From \cite[Theorem 5.1.35]{eng} it follows that $Y$ is paracompact, and so, from \cite[Theorem 5.1.5]{eng} it follows that $Y$ is normal. 

    Since $Y$ is an M-space, from \cite{Mich} it follows that
    $$
        \forall y \in Y [y \text{ is a q-point}].
    $$
    Now from Lemma \ref{bound_is_comp} it follows that 
    $$
        \forall y \in Y[\partial_{\sigma(\tau, f)}f^{-1}(y) \text{ is a compact subset of } \langle X, \sigma(\tau, f) \rangle].
    $$
    Let $L$ be the map of $Y$ to $\sigma(\tau, f)$ such that
    \[
        L(y) \coloneq 
        \begin{cases}
        \mathsf{Int}_{\sigma(\tau, f)}(f^{-1}(y)), & \text{if } \partial_{\sigma(\tau, f)}f^{-1}(y) \neq \varnothing \\
        f^{-1}(y) \setminus \{p_y\}, & \text{else}
        \end{cases}
    \]
    for all $y \in Y$, where $p_y$ is an arbitrary point of $f^{-1}(y)$. Then $X_0 \coloneq X \setminus \bigcup_{y \in Y}L(y)$ is a closed subset of $\langle X, \sigma(\tau, f) \rangle$ and $f \upharpoonright X_0$ is a perfect map from $\langle X_0, \sigma(\tau, f) \rangle$ onto $Y$. From \cite[Corollary 3.7.3]{eng} it follows that $g\circ f$ is a perfect map from $\langle X_0, \sigma(\tau, f) \rangle$ onto $M(Y)$, and so $\langle X_0, \sigma(\tau, f) \rangle$ is an M-space. Since $\langle X_0, \sigma(\tau, f) \rangle$ is submetrizable, then from \cite[Corollary 3.8]{gms} it follows that $\langle X_0, \sigma(\tau, f) \rangle$ is metrizable, and so, since $g \circ f$ is perfect, from \cite[Theorem 4.4.15]{eng} it follows that $Y$ is metrizable. 
\end{proof}

\begin{teo}
    Let $X$ be submetrizable, $Y$ functionally Hausdorff, dense-in-itself, and $\psi(Y) = \omega$, and $f\colon X \to Y$ open-closed. Then $Y$ is submetrizable. 
\end{teo}

\begin{proof}
    Fix a topology $\tau \in \tau(X)\downarrow_{metr}$ and notice that 
    $$
        f\colon \langle X, \sigma(\tau, f) \rangle \to Y \text{ is continuous and open-closed.}
    $$
    From Proposition \ref{Eng_lemma} it follows that
    \begin{equation}
        \forall y \in Y [\text{any HCP collection of neighbourhoods of } y \text{ is finite}].
    \end{equation}
    Then from Lemma \ref{heart} it follows that
    \begin{equation} \label{cl_comp}
        \forall y \in Y\ \forall \gamma \in \sigma(\tau, f)\downarrow_{metr}[\mathsf{Cl}_{\gamma}(f^{-1}(y)) \text{ is a compact subset of } \langle X, \gamma \rangle].
    \end{equation}
    Now, by Lemma \ref{criter_for_rel_pseud}, $f^{-1}(y)$ is a relatively pseudocompact subset of $\langle X, \sigma(\tau, f) \rangle$ for any $y \in Y$.

    Therefore from Lemmas \ref{lemma_1} and \ref{Utv_4} we can conclude that
    $$
        \mathsf{Cl}_{\tau}(f^{-1}(y)) = \mathsf{Cl}_{\tau}(\partial_{\sigma(\tau, f)}f^{-1}(y)) \subseteq f^{-1}(y).
    $$  
    And so, since $\sigma(\tau, f) \upharpoonright f^{-1}(y) = \tau \upharpoonright f^{-1}(y)$, From \ref{cl_comp} and Proposition \ref{miz_theorem} it follows that $Y$ is submetrizable.
\end{proof}

\begin{teo}
    Let $X$ be submetrizable, $Y$ functionally Hausdorff and dense-in-itself with the following property: $\forall y\in Y\ \exists Q(y) \subseteq Y\ [y \text{ is a non-isolated q-point in } Q(y)]$, and $f\colon X \to Y$ open-closed. Then $Y$ is submetrizable. 
\end{teo}

\begin{proof}
    Fix a topology $\tau \in \tau(X)\downarrow_{metr}$ and notice that 
    $$
        f\colon \langle X, \sigma(\tau, f) \rangle \to Y \text{ is continuous and open-closed.}
    $$
    We prove that $f$ is perfect. Take $y \in Y$ and $Q(y) \subseteq Y$. Note that
    \begin{equation} \label{AAAAAAAAAA}
        \sigma(\tau, f) \upharpoonright f^{-1}[Q(y)] = \sigma(\tau \upharpoonright f^{-1}[Q(y)], f \upharpoonright f^{-1}[Q(y)]).
    \end{equation}
    Let $g \coloneq f \upharpoonright f^{-1}[Q(y)]$ and let $\pi \coloneq \sigma(\tau, f) \upharpoonright f^{-1}[Q(y)]$. Then
    $$
        g \colon \langle f^{-1}[Q(y)], \tau(X) \rangle \to Q(y) \text{ is continuous and open-closed,}
    $$
    $$
        \tau\upharpoonright f^{-1}[Q(y)] \in (\tau(X) \upharpoonright f^{-1}[Q(y)])\downarrow_{metr},
    $$
    $$
        \langle Q(y), \tau(Y) \rangle \text{ is functionally Hausdorff},
    $$
    and
    $$
        y \text{ is a q-point in } Q(y).
    $$
    So, from Lemma \ref{bound_is_comp} and (\ref{AAAAAAAAAA}) it follows that 
    $$
        \partial_{\pi}g^{-1}(y) \text{ is a compact subset of } \langle f^{-1}[Q(y)], \pi \rangle.
    $$
    Since $y$ is non-isolated in $Q(y)$ and $g$ is open, we see that $\partial_{\pi}g^{-1}(y) = g^{-1}(y) = f^{-1}(y)$. And so, $f^{-1}(y)$ is a compact subset of $\langle X, \sigma(\tau, f) \rangle$. Now, the statement of the Theorem follows from Proposition \ref{miz_theorem}.
\end{proof}

\begin{corr}
    Let $X$ be submetrizable, $Y$ functionally Hausdorff and dense-in-itself, and $f\colon X \to Y$ open-closed. If one of the following conditions holds, then $Y$ is submetriable:
    \begin{itemize}
        \item $Y$ is a q-space; 
        \item for any $y \in Y$ there exists a nontrivial sequence which is converges to $y$.
    \end{itemize}
\end{corr}

The following theorem shows that we cannot omit dense-in-itself condition in previous results.

\begin{teo}
    There exist a submetrizable space $X$, a regular hereditarily paracompact non submetrizable first-countable space $Y$, and an open-closed $f\colon X \to Y$.
\end{teo}

\begin{proof}
    Consider the space $\mathsf{D}(2) \times \mathbb{M}$ ($\mathsf{D}(2)$ is the set $2$ with the discrete topology), it is a free sum of the Michael line with itself. Also consider the following equivalence relation on $\mathsf{D}(2) \times \mathbb{M}$:
    $$
        \sim\ \coloneq \{\langle x, x\rangle\ \colon\ x\in \mathsf{D}(2) \times \mathbb{M}\} \cup \{\big\langle \langle n,q\rangle, \langle k, q\rangle \big\rangle\ \colon\ q \in \mathbb{Q} \text{ and } n,k\in 2\}.
    $$
    In \cite{popov} Popov proved that $\faktor{\mathsf{D}(2)\times\mathbb{M}}{\sim}$ is hereditarily paracompact and does not have a $G_\delta$-diagonal, in particular it is not submetrizable. It is also easy to see that $\faktor{\mathsf{D}(2)\times\mathbb{M}}{\sim}$ is first-countable. Let $g\colon \mathsf{D}(2) \times \mathbb{M} \to \faktor{\mathsf{D}(2)\times\mathbb{M}}{\sim}$ be the corresponding quotient map. Note that 
    \begin{equation} \label{base_at_q}
        \forall q \in \mathbb{Q}[\{g[\mathsf{D}(2) \times (q-\frac{1}{n}, q + \frac{1}{n})]\ \colon\ n \in \omega\} \text{ is a base at } g(\langle 0, q\rangle)].
    \end{equation}

    Now let us define the new topology $\tau$ on $\mathbb{R}\times\mathbb{I}$ by the base $\tau_{\mathbb{R}\times\mathbb{I}} \cup \{\{x\} \colon x(0) \in \mathbb{P}\}$. It is obvious that $\langle \mathbb{R}\times\mathbb{I}, \tau \rangle$ is submetrizable. Fix disjoint dense subsets $I_0$ and $I_1$ of $\mathbb{I}$ such that $\mathbb{I} = I_0 \cup I_1$. Let $f$ be the map of $\langle \mathbb{R}\times\mathbb{I}, \tau \rangle$ on $\faktor{\mathsf{D}(2)\times\mathbb{M}}{\sim}$ such that 
    \[
        f(x) \coloneq 
        \begin{cases}
        g(\langle 0, x(0) \rangle), & x(1) \in I_0 \\
        g(\langle 1, x(0) \rangle), & x(1) \in I_1 \\
        \end{cases}
    \]
    We prove that $f$ is continuous and open-closed. 

    Take $a < b \in \mathbb{R}$ and $c < d \in \mathbb{I}$. Then
    \[
    \begin{split}
        f[(a,b)\times(c,d)]  &= f[(a,b)\times((c,d) \cap I_0)] \cup f[(a,b)\times((c,d) \cap I_1)]=\\ &= g[\{0\}\times(a,b)] \cup g[\{1\}\times(a,b)]=\\
        &= g[\mathsf{D}(2)\times(a,b)]
    \end{split}
    \]
    is open. And so, from (\ref{base_at_q}) it follows that $f$ is open and continuous. 

    Let $F$ be a closed subset of $\langle \mathbb{R}\times\mathbb{I}, \tau \rangle$. We prove that $f[F]$ is closed. Take a nontrivial convergent sequence $\langle y_n \rangle_{n\in\omega}$ in $f[F]$. Suppose that $y$ is a limit of $\langle y_n \rangle_{n\in\omega}$. Notice that $y = g(\langle 0, q \rangle) = g(\langle 1, q\rangle)$ for some $q \in \mathbb{Q}$, because otherwise $y$ is isolated.  For every $k \in \omega$ fix $y_{n(k)} \in g[\mathsf{D}(2) \times (q-\frac{1}{k}, q + \frac{1}{k})]$. Then $\langle y_{n(k)}\rangle_{k \in \omega}$ converges to $y$ and 
    \begin{equation} \label{fib_sub_int}
        \forall k \in \omega [f^{-1}(y_{n(k)}) \subseteq (q-\frac{1}{k}, q + \frac{1}{k}) \times \mathbb{I}].
    \end{equation}
    Now for every $k \in \omega$ fix $x_k \in f^{-1}(y_{n(k)}) \cap F$. Then there exists $\langle x_{k(m)} \rangle_{m \in \omega}$ a subsequent of $\langle x_k\rangle_{k \in \omega}$ such that $\langle x_{k(m)} \rangle_{m \in \omega}$ has a limit $x$ in $\mathbb{R}\times\mathbb{I}$ with its usual topology. From (\ref{fib_sub_int}) it follows that $x \in \{q\}\times \mathbb{I} = f^{-1}(y) $, and so $x(0) \in \mathbb{Q}$, but then $x$ is a limit point of $\langle x_{k(m)} \rangle_{m \in \omega}$ in $\langle \mathbb{R}\times\mathbb{I}, \tau \rangle$. Therefore $f(x) = y \in f[F]$.
\end{proof}

\section{Unanswered questions}

\begin{ques}
    If a pseudocompact space $X$ is a closed image of a submetrizable space, must $X$ be metrizable?
\end{ques}

\begin{ques}
    Let $X$ be an open-closed image of a submetrizable space. If $X$ has a $G_\delta$-diagonal, is $X$ submetrizable?
\end{ques}

\begin{ques}
    What kinds of internal characterizations of the closed images of
    submetrizable spaces?
\end{ques}

\begin{ques}
    What kinds of internal characterizations of the open-closed images of
    submetrizable spaces?
\end{ques}

\end{document}